\documentclass[a4paper,11pt]{article}
\usepackage{fullpage}
\usepackage{dsfont}
\usepackage[utf8]{inputenc}
\usepackage[english]{babel} 
\usepackage{authblk}  

\usepackage{enumitem}
\setlist[itemize]{noitemsep}

\usepackage{amsfonts, amssymb, amsbsy, amsmath, amsthm}

\usepackage{mathrsfs}
\usepackage{stmaryrd}
\usepackage{bbm}

\usepackage{array}
\usepackage{siunitx}
\newcommand{\rating}[1]{%
  \ifnum#1>0$\bullet$\else$\circ$\fi%
  \ifnum#1>1$\bullet$\else$\circ$\fi%
  \ifnum#1>2$\bullet$\else$\circ$\fi%
  \ifnum#1>3$\bullet$\else$\circ$\fi%
  \ifnum#1>4$\bullet$\else$\circ$\fi%
}

\allowdisplaybreaks

\usepackage[sort,numbers]{natbib}
\usepackage[hyphens]{url}
\usepackage{xcolor}
\usepackage[textsize=small]{todonotes}
\definecolor{darkblue}{rgb}{0,0,0.7}
\usepackage[colorlinks=true,allcolors=darkblue]{hyperref}
\usepackage[capitalise]{cleveref}
\crefname{algocf}{Algorithm}{Algorithms}
\crefname{equation}{Equation}{Equations}
\crefname{figure}{Figure}{Figures}
\crefname{enumi}{Item}{Items}
\usepackage{mathtools}
\usepackage{nicefrac}

\usepackage{tikz}
\usepackage{tikz-cd}
\usetikzlibrary{automata,arrows,positioning,calc,decorations.pathmorphing}
\usetikzlibrary{patterns}
\usetikzlibrary{patterns.meta}
\usepackage{caption}
\newtheorem{theorem}{Theorem}[section]
\newtheorem{lemma}[theorem]{Lemma}

\newtheorem{definition}[theorem]{Definition}
\newtheorem{corollary}[theorem]{Corollary}

\newtheorem{proposition}[theorem]{Proposition}
\theoremstyle{definition}

\newtheorem{remark}[theorem]{Remark}

\newcounter{questioncounter}

\newtheorem*{sidconj}{Sidorenko's Conjecture}
\newtheorem*{forcconj}{Forcing Conjecture}

\newcommand{\NN}{\mathbb{N}}

\makeatletter
\newcommand\@rest[3]{%
        \begingroup
                #2 %
                \sbox0{$\m@th#2\left.\kern-\nulldelimiterspace\vphantom{#1}\littletaller\right|$}%
                \sbox2{$\m@th#2\upharpoonright$}%
                \dimen0=\ht0 %
                \advance\dimen0 by -\ht2 %
                \begingroup
                        #2#1%
                \endgroup
                \mathclose{%
                \kern\nulldelimiterspace\mathclap{\usebox0}\mathclap{\raisebox{\dimen0}{\usebox2}}\kern\nulldelimiterspace}_{#3}
        \endgroup
}
\newcommand\rest[2]{%
        \mathpalette{\@rest{#1}}{#2}%
}

\newcommand\restrict[2]{
   \left.\kern-\nulldelimiterspace 
   #1
   \littletaller 
   \right|_{#2}%
   }
 \newcommand{\littletaller}{\mathchoice{\vphantom{\big|}}{}{}{}}
\makeatother

\newcommand{\defemph}[1]{\emph{#1}}

\title{Adjunctions, Box Products, and Forcing Families}

\author[1,3]{Aldo Kiem\thanks{Supported by the Deutsche Forschungsgemeinschaft (DFG, German Research Foundation) under Germany's Excellence Strategy – The Berlin Mathematics Research Center MATH+ (EXC-2046/1, project ID: 390685689).}}
\author[1,2]{Olaf Parczyk}
\author[1,3]{Christoph Spiegel}
\date{}
\hypersetup{
    pdftitle={Adjunctions, Box Products, and Forcing Families},
    pdfauthor={Aldo Kiem, Olaf Parczyk, and Christoph Spiegel}
}

\affil[1]{\small Zuse Institute Berlin, Department AIS2T, \emph{lastname}@zib.de}
\affil[2]{\small Freie Universit\"at Berlin, Institute of Mathematics}
\affil[3]{\small Technische Universit\"at Berlin, Institute of Mathematics}

\begin{document}

\maketitle

\begin{abstract}
    Sidorenko's conjecture states that the number of copies of any given bipartite graph in another graph of given density is asymptotically minimized by a random graph.
    For bipartite graphs containing a cycle, the forcing conjecture further asserts that asymptotic equality characterizes quasi-random graphs.
    We establish an adjoint identity for a general class of graph-substitution operators and use it to obtain Sidorenko and forcing results for balanced blow-ups, subdivisions, Cartesian products, and strong products.
\end{abstract}

\section{Introduction}

Sidorenko's conjecture and the closely related forcing conjecture are two fundamental open problems in extremal graph theory.
For graphs $G$ and $H$, let $t(G, H)$ denote the probability that a random map $\phi: V_G \rightarrow V_H$ is a graph homomorphism from $G$ to $H$, i.e.~for all edges $\{u, v\} \in E_G$, we have $\{\phi(u),\phi(v)\} \in E_H$.
For a graph $G$, write $v_G=\lvert V_G\rvert$ and $e_G=\lvert E_G\rvert$.
For $k\ge0$, let $I_k$ denote the edgeless graph on $k$ vertices.
Sidorenko~\cite{1993:Sidorenko:CorrelationInequalityBipartiteGraphs}, and independently in an earlier paper Erd\H{o}s and Simonovits~\cite{1984:Simonovits:ExtremalGraphProblems,1984:ErdosSimonovits:SupersaturatedGraphs}, formulated the following conjecture.
\begin{sidconj}
    For every bipartite graph $G$ and every graph $H$, we have
    \begin{equation} \label{eq:Sidorenko}
        t(G,H) \ge t(K_2,H)^{e_G}.
    \end{equation}
\end{sidconj}

As the right hand side of \cref{eq:Sidorenko} is the asymptotic expected normalized homomorphism density of $G$ in a binomial random graph of density $t(K_2, H)$, the conjecture states that $t(G, H)$ is minimized by that random graph.
If $G$ is not bipartite, i.e.~it contains an odd cycle, then any bipartite graph $H$ with positive density shows that \cref{eq:Sidorenko} cannot hold.

We survey part of the history of the conjecture and discuss the related forcing conjecture in the next section, which readers familiar with this history may skip. Our results are covered in the subsequent section.

\subsection{History of Sidorenko's conjecture and the forcing conjecture}
This conjecture has been verified for large families of graphs.
We call a graph $G$ \defemph{Sidorenko} if \cref{eq:Sidorenko} holds for every $H$.
Sidorenko~\cite{1993:Sidorenko:CorrelationInequalityBipartiteGraphs} already provided evidence for his conjecture by showing that trees, even cycles, and complete bipartite graphs are Sidorenko, whereas paths were already resolved before the conjecture was stated~\cite{1959:MulhollandSmith:ArisingGenetical,1965:BlakleyRoy:HolderTypeInequalitySymmetric}.
Other families of graphs known to be Sidorenko include cubes~\cite{2010:Hatami:GraphNormsSidorenko}, bipartite graphs with one vertex complete to the other side~\cite{2010:ConlonEtAl:ApproximateSidorenko}, the Cartesian product $T \Box H$ of a tree $T$ with a Sidorenko graph $H$~\cite{2016:KimEtAl:TwoApproachesSidorenko}, and the Cartesian product $C_{2k} \Box H$ of an even cycle with a Sidorenko graph $H$~\cite{2018:ConlonEtAl:SomeAdvancesSidorenko}.
Certain tree-like graphs were also shown to be Sidorenko in~\cite{2016:KimEtAl:TwoApproachesSidorenko,2011:LiSzegedy:LogarithmicCalculusSidorenko}; this was superseded by Conlon, Kim, Lee, and Lee~\cite{2018:ConlonEtAl:SomeAdvancesSidorenko}, who showed that strongly tree-decomposable graphs are Sidorenko.
Li and Szegedy~\cite{2011:LiSzegedy:LogarithmicCalculusSidorenko} and Szegedy~\cite{2014:Szegedy:InformationTheoreticSidorenko} developed a recursive procedure to construct Sidorenko graphs from smaller structures, which covered all results known at that point but admits no simple non-recursive description.
Two more recent results break out of this pattern:
Conlon and Lee~\cite{2018:ConlonLee:SidorenkoBlowUps} proved that every bipartite graph has a one-sided blow-up that is Sidorenko, which also implies that all sufficiently large balanced blow-ups are Sidorenko; and
Im, Li, and Liu~\cite{2024:ImEtAl:SidorenkoSubdivisionsTheta} showed that specific subdivisions, possibly starting from an arbitrary graph, are Sidorenko.
Conlon, Kim, Lee, and Lee~\cite{2018:ConlonEtAl:SomeAdvancesSidorenko} also connected Sidorenko's conjecture to the Kohayakawa--Nagle--R\"odl--Schacht (KNRS) conjecture~\cite{2010:KohayakawaEtAl:WeakHypergraphRegularity}: if $H$ satisfies the KNRS conjecture, then its $1$-subdivision is Sidorenko.
This was subsequently extended by Chen, Lin, and Ma~\cite{2024:ChenEtAl:KNRSSubdivisions}, who proved that if $H$ satisfies the KNRS conjecture, then its $(2k-1)$-subdivision is Sidorenko for all $k \ge 1$.
Many of the above results are due to complementary approaches to Sidorenko's conjecture, including analytic methods based on graph norms and H\"older-type inequalities~\cite{2010:Hatami:GraphNormsSidorenko,2018:ConlonEtAl:SomeAdvancesSidorenko}, and structural decomposition approaches such as reflection and cut-percolation frameworks and their variants~\cite{2017:ConlonLee:FiniteReflectionGroupsGraphNorms,2022:Coregliano:LeftCutPercolationInducedSidorenko,2014:Szegedy:InformationTheoreticSidorenko}.
Our results likewise concern blow-ups and subdivisions. Before stating them, we discuss the forcing conjecture.

A sequence of graphs $(H_n)_{n \in \NN}$ of increasing order is \defemph{$p$-quasi-random} if for every graph $G$ we have $t(G, H_n)  = \big( 1 + o(1) \big) \, p^{e_G}$.
The concept of quasi-random graphs was established by Chung, Graham, and Wilson~\cite{1989:ChungEtAl:QuasiRandomGraphs}, who also gave several equivalent definitions.
In particular, $(H_n)_{n \in \NN}$ is $p$-quasi-random if and only if $t(K_2, H_n) = \big( 1 + o(1) \big) \, p$ and $t(C_4, H_n) = \big( 1 + o(1) \big) \, p^{4}$.
They asked which graphs can replace $C_4$ in this characterization.
Following Conlon, Fox, and Sudakov~\cite{2010:ConlonEtAl:ApproximateSidorenko}, for $p\in(0,1)$ we call a graph $G$ \defemph{$p$-forcing} if every sequence of graphs $(H_n)_{n \in \NN}$ with $t(K_2, H_n)=\big( 1 + o(1) \big) \, p$ and $t(G, H_n)=\big( 1 + o(1) \big) \, p^{e_G}$ is $p$-quasi-random; $G$ is \defemph{forcing} if it is $p$-forcing for every $p\in(0,1)$.
A forcing graph must be bipartite and contain at least one cycle~\cite{2010:ConlonEtAl:ApproximateSidorenko}.
Skokan and Thoma~\cite{2004:SkokanThoma:BipartiteSubgraphsQuasiRandomness} asked whether these are the only exceptions, and Conlon, Fox, and Sudakov~\cite{2010:ConlonEtAl:ApproximateSidorenko} conjectured that this is the case.
\begin{forcconj}
    Every bipartite graph with at least one cycle is forcing.
\end{forcconj}
The conjecture is known for weakly norming graphs containing a cycle~\cite{2010:Hatami:GraphNormsSidorenko,2010:ConlonEtAl:ApproximateSidorenko}.
Beyond this norm-based setting, Conlon, Fox, and Sudakov proved it for bipartite graphs having two vertices in one part complete to the opposite part, which has at least two vertices~\cite[Theorem~1.3]{2010:ConlonEtAl:ApproximateSidorenko}, and Li and Szegedy extended this to every non-tree bipartite graph having one such vertex~\cite[Theorem~2.4]{2011:LiSzegedy:LogarithmicCalculusSidorenko}.
Conlon, Fox, and Sudakov~\cite{2010:ConlonEtAl:ApproximateSidorenko} note that this is stronger than Sidorenko's conjecture, and that for $G$ to be forcing one must show that the binomial random graph is the sole minimizer of $t(G, H)$ among graphs of the same density.
When a graph $G$ is proved to be Sidorenko, close inspection of the proof often reveals that it is also forcing if equality in the chain of inequalities propagates to a known forcing graph.

A natural question is whether the edge density $t(K_2, H_n)$ can be replaced by other subgraph densities.
More generally, a family of graphs $\mathcal{F}$ is \defemph{$p$-forcing} if every sequence of graphs $(H_n)_{n \in \NN}$ satisfying $t(F, H_n)=\big( 1 + o(1) \big) \, p^{e_F}$ for all $F \in \mathcal{F}$ is $p$-quasi-random, and it is \defemph{forcing} if it is $p$-forcing for every $p\in(0,1)$~\cite{2010:ConlonEtAl:ApproximateSidorenko}.
A forcing family with two members is called a \emph{forcing pair}.
If $H$ is forcing, then $\{K_2,H\}$ is a forcing family by definition.
Chung, Graham, and Wilson~\cite[p.~361]{1989:ChungEtAl:QuasiRandomGraphs} listed the pairs $\{K_{2,s},K_{2,t}\}$, where $s\ne t$ and $s,t\ge 2$, among further examples of forcing families.\footnote{In the same paragraph, they also erroneously list $\{C_{2s},C_{2t}\}$, $s\ne t$, as a forcing family, which was later implicitly corrected by Shapira and Tyomkyn~\cite[Example~2 and p.~633]{2021:ShapiraTyomkyn:QuasirandomPantograph}.}
The first forcing family involving a non-bipartite graph was found by Conlon, H\`an, Person, and Schacht~\cite{2012:ConlonEtAl:WeakQuasiRandomnessHypergraphs}, who exhibited the pair $(K_3,M(3))$, where $M(3)$ is the line graph of the three-dimensional hypercube $Q_3$. Reiher and Schacht~\cite{2019:ReiherSchacht:ForcingQuasirandomness} later showed that $K_3$ together with the triangle-replacement of a forcing graph $G$ is a forcing pair.
H\`an, Person, and Schacht~\cite{2011:HanEtAl:NoteForcingPairs} first systematically studied forcing pairs and proved that every graph $F$ with at least one edge has a partner $G$ such that $(F,G)$ is a forcing pair.

\subsection{New forcing families}

We now state our results, postponing the technical setup and our main theorem on adjoints to \Cref{sec:adjoint}.
The propositions below apply the framework to specific graph operations; the forcing-family preservation arguments are developed in \Cref{sec:additional_results}, in the spirit of the triangle-replacement result of Reiher and Schacht~\cite{2019:ReiherSchacht:ForcingQuasirandomness} mentioned above.

The \emph{$m$-fold blow-up} $G[m]$ of a graph $G$ has vertex set $V_G \times [m]$, with $(v, i)$ and $(w, j)$ adjacent whenever $v$ and $w$ are adjacent in $G$.

\begin{proposition} \label{thm:blowup}
    If $G$ is Sidorenko and has at least one edge, then the $m$-fold blow-up $G[m]$ is forcing for every $m \ge 2$.
    Moreover, forcing families are preserved by blow-ups, i.e.~if $\mathcal F$ is forcing, then so is $\{F[m]\mid F \in \mathcal F\}$ for every $m \ge 2$.
\end{proposition}

Let $G$ have bipartition $A\sqcup B$, and write $G_A^q$ for the one-sided blow-up that keeps $A$ fixed and replaces every vertex of $B$ by $q$ twins.
As a corollary of a more general result, Conlon and Lee~\cite[Corollary~1.3]{2018:ConlonLee:SidorenkoBlowUps} showed that $G_A^{m_0}$ is Sidorenko for some positive integer $m_0$.
This implies that $G[m]$ is Sidorenko for every $m\ge m_0$, as also stated in passing in~\cite{2024:BaekEtAl:ExtremalGraphs}.\footnote{For completeness, here is the short argument.
For a host graph $H$, let $X$ be the conditional homomorphism density of $G$ obtained by fixing the images of $A$.
Then monotonicity of $L^r$-norms gives
\[
    t(G_A^m,H)^{1/m}=(\mathbb E X^m)^{1/m}\ge(\mathbb E X^{m_0})^{1/m_0}=t(G_A^{m_0},H)^{1/m_0}\ge t(K_2,H)^{e_G}.
\]
Since $G[m]$ is obtained from $G_A^m$ by cloning every vertex of its $A$-side $m$ times, the same conditioning identity and Jensen's inequality give $t(G[m],H)\ge t(G_A^m,H)^m\ge t(K_2,H)^{m^2e_G}.$
}

Our result makes the stronger assumption that the base graph is already Sidorenko, but in return shows that every balanced blow-up with $m\ge2$ is forcing.
This first assertion also has the following direct proof.
For a graph $H$, define the auxiliary graph $H'$ on $V_H^m$ by joining two distinct tuples $(y_1,\dots,y_m)$ and $(z_1,\dots,z_m)$ whenever $\{y_i,z_j\}\in E_H$ for every
$i,j\in[m]$.
Then a short calculation gives $t(G[m],H)=t(G,H')$ for any graph $G$, and the Sidorenko properties of $G$ and $K_2[m]=K_{m,m}$ give
\begin{align}
\label{eq:chain_intro}
   t(G[m],H)=t(G,H') \ge t(K_2,H')^{e_G} = t(K_2[m],H)^{e_G} \ge t(K_2,H)^{m^2e_G}.
\end{align}
To verify the equality case, let $(H_n)_{n\in\NN}$ satisfy $t(K_2,H_n)=(1+o(1))p$ and $t(G[m],H_n)=(1+o(1))p^{m^2e_G}$.
The first and last terms of \cref{eq:chain_intro} then have the same limit, so $t(K_{m,m},H_n)\to p^{m^2}$.
Since $K_{m,m}$ is forcing for $m\ge2$~\cite{2004:SkokanThoma:BipartiteSubgraphsQuasiRandomness}, the sequence $(H_n)_{n\in\NN}$ is $p$-quasi-random, proving that $G[m]$ is forcing.
This proof illustrates the general pattern developed in this paper.
We introduce an adjunction framework that captures this auxiliary-graph construction for a broad class of graph operations, separating the common argument from the checks specific to each operation.
The Cartesian and strong products provide two further examples.

The \emph{box product} (also called Cartesian product) $G \Box F$ of two graphs $G$ and $F$ has vertex set $V_G \times V_F$, with two vertices $(u,v)$ and $(u',v')$ adjacent if $u=u'$ and $\{v,v'\} \in E_F$, or if $v=v'$ and $\{u,u'\} \in E_G$.
For example, $C_4=K_2\Box K_2$; set $Q_0=K_1$ and, for $n\ge1$, define the $n$-dimensional hypercube by $Q_n=Q_{n-1}\Box K_2$.
\begin{proposition}\label{prop:Box_Product_Preserving}
    Let $H$ be a bipartite graph with at least one edge and no isolated vertices, and let $\mathcal F$ be a forcing family. Then, $\{F\Box H\mid F \in \mathcal F \}\cup \{H\}$ is a forcing family.
\end{proposition}

Similarly, the \emph{strong product} $G \boxtimes F$ of two graphs $G$ and $F$ has vertex set $V_G \times V_F$, where $(u,v)$ is adjacent to $(u',v')$ if either $u=u'$ and $\{v,v'\}\in E_F$, or $v=v'$ and $\{u,u'\}\in E_G$, or both $\{u,u'\}\in E_G$ and $\{v,v'\}\in E_F$.
For example, $K_1 \boxtimes G = G$ for every graph $G$, and $K_2 \boxtimes K_2 = K_4$.
\begin{proposition}\label{prop:Strong_Product_Preserving}
    Let $H$ be a bipartite graph with at least one edge and no isolated vertices, and let $\mathcal F$ be a forcing family. Then, $\{F\boxtimes H\mid F \in \mathcal F \}\cup \{H\}$ is a forcing family.
\end{proposition}
Previous results show that $T\Box G$ is Sidorenko when $T$ is a tree and $G$ is Sidorenko~\cite[Theorem~1.2]{2016:KimEtAl:TwoApproachesSidorenko}, and that $C_{2k}\Box G$ is Sidorenko for $k\ge2$ and Sidorenko $G$~\cite[Theorem~1.4]{2018:ConlonEtAl:SomeAdvancesSidorenko}.
We are not aware of a previous forcing-family preservation theorem for Cartesian products, or of any previous result on the Sidorenko or forcing properties of strong products.

For integers $k\ge1$, the \emph{$k$-subdivision} of a graph $H$ is obtained by replacing every edge of $H$ by a path of length $k+1$, i.e.~by inserting $k$ new vertices on every edge.

\begin{proposition}
\label{thm:subdivisions}
    The $k$-subdivision of a forcing family $\mathcal F$ is again a forcing family.
\end{proposition}
Conlon, Kim, Lee, and Lee proved that the $1$-subdivision of every Sidorenko graph is Sidorenko~\cite[Theorem~1.3(ii)]{2018:ConlonEtAl:SomeAdvancesSidorenko}.
Conlon and Lee proved that the $1$-subdivisions of complete graphs and of the $1$-skeleta of regular polytopes are weakly norming~\cite[Theorem~1.1]{2017:ConlonLee:FiniteReflectionGroupsGraphNorms}; whenever they contain a cycle, they are therefore forcing~\cite{2010:ConlonEtAl:ApproximateSidorenko}.
We are not aware of an earlier theorem asserting that subdivision preserves arbitrary forcing families.

In \Cref{prop:Box_Product_Preserving,prop:Strong_Product_Preserving}, the graph $H$ must be added to the family, while this is not necessary in \Cref{thm:blowup,thm:subdivisions}.
We present one additional result for the box product that provides new forcing graphs and generalizes \cite[Theorem 1.4]{2018:ConlonEtAl:SomeAdvancesSidorenko}, which states that, for every integer $k\ge2$, $G\Box C_{2k}$ is Sidorenko whenever $G$ is Sidorenko.
A graph $F$ is \defemph{dominated} by another graph $G$ if $t(G,H)^{e_F} \ge t(F,H)^{e_G}$ for every graph $H$; see Conlon and Lee~\cite{2024:ConlonLee:DominationInequalitiesGraphs} for a thorough discussion of domination inequalities.

\begin{proposition}
    \label{thm:box-product}
    For $n\ge0$, let $G$ and $F$ be Sidorenko graphs with at least one edge, and suppose that $F$ is dominated by the hypercube $Q_{n+2}$.
    Then the triple box product $G\Box F\Box Q_n$ is forcing.
\end{proposition}

In particular, we can take $F$ to be any Sidorenko subgraph of $Q_{n+2}$, since hypercubes are weakly norming and weakly norming graphs dominate all their subgraphs with at least one edge~\cite{2010:Hatami:GraphNormsSidorenko,2024:ConlonLee:DominationInequalitiesGraphs}.
Indeed, for every $k\ge2$ and graph $H$, we have $t(C_{2k},H)^{1/(2k)}\le t(C_4,H)^{1/4}$~\cite{2010:Hatami:GraphNormsSidorenko,2017:ConlonLee:FiniteReflectionGroupsGraphNorms}, so $C_{2k}$ is dominated by $Q_2=C_4$. Setting $n=0$ and $F=C_{2k}$ therefore recovers the Sidorenko conclusion of \cite[Theorem 1.4]{2018:ConlonEtAl:SomeAdvancesSidorenko} and strengthens it to forcing.

\medskip
\noindent\textbf{Outline.} \Cref{sec:adjoint} proves the adjoint theorem and derives the first assertion of \Cref{thm:blowup}.
\Cref{sec:proof-box-product} proves \Cref{thm:box-product}, while \Cref{sec:additional_results} proves the second assertion of \Cref{thm:blowup} and \Cref{prop:Box_Product_Preserving,prop:Strong_Product_Preserving,thm:subdivisions}.
\Cref{sec:hypergraphs} discusses an extension to hypergraphs, and \Cref{sec:concluding} contains concluding remarks.

\section{The technical adjoint theorem}
\label{sec:adjoint}

Given a probability space $\Omega$, a graphon on $\Omega$ is a symmetric measurable function $W : \Omega^2 \rightarrow [0,1]$; graphons are considered up to weak isomorphism. They were introduced as limit objects for graphs~\cite{2006:LovaszSzegedy:LimitsDenseGraphSequences}.
Every graphon can be represented on $[0,1]$ with Lebesgue measure (see~\cite{2012:Lovasz:LargeNetworksGraphLimits} for an extensive treatment of the subject), and we assume this representation unless specified otherwise.
For a graph $G$ and a graphon $W$, its homomorphism density is
\[
    t(G,W)=\int_{[0,1]^{V_G}}\prod_{uv\in E_G}W(y_u,y_v)\, \mathrm d \mathcal L(y).
\]
A \emph{formal linear combination of graphs} is a finite sum $g=\sum_{i=1}^k a_iG_i$ with $a_i\in\mathbb R$; we use the same terminology for vertex-colored graphs.
We extend homomorphism densities and the graph operators below linearly, and identify a scalar with that multiple of the empty graph.
We write $G^m$ for the disjoint union of $m$ copies of $G$.

The sequence definition of forcing has the following equivalent graphon formulation: for every $p\in(0,1)$, a family $\mathcal F$ is $p$-forcing if and only if every graphon $W$ satisfying $t(F,W)=p^{e_F}$ for all $F\in\mathcal F$ is equal almost everywhere to the constant-$p$ graphon~\cite[Section~16.7.1]{2012:Lovasz:LargeNetworksGraphLimits}.

Our proofs require vertex colors.
By a \emph{vertex-colored graph} we mean a graph whose vertices carry one of two colors.
All vertex-colored graphs used below have every vertex in one distinguished \emph{main color}; the second color is needed only on the graphon side.
Similarly, a \emph{vertex-colored graphon} is a tuple $(W,f)$ where $W$ is a graphon on $\Omega$ and $f:\Omega\to [0,1]$ encodes the vertex colors; formally, the complement $1-f$ represents the second color.
For a color function $f$, write $A_f=\{x\in\Omega\mid f(x)>0\}$.
Given a vertex-colored graph $G$ (all vertices in the main color) and a vertex-colored graphon $(W,f)$, the density of $G$ in $(W,f)$ is defined as
\begin{equation*}
    t(G,(W,f)) = \int_{[0,1]^{V_G}}\prod_{v\in V_G} f(y_v)\prod_{uv\in E_G} W(y_u,y_v) \, \mathrm d \mathcal{L} (y).
\end{equation*}
See \cref{fig:vertex-colored-graphon} for an example.
Note the following: When $G$ is a graph without isolated vertices, then it is Sidorenko if and only if $t(I_{2e_G-v_G}\sqcup G,(W,f)) \ge t(K_2,(W,f))^{e_G}$ for any vertex-colored graphon.
Moreover, if $G$ is forcing (and thus not a matching), then equality together with $t(K_2,(W,f))>0$ implies that the graphon $W$ is constant almost everywhere on $A_f^2$.
For notational convenience: if $G$ and $H$ are graphs such that $t(G,W)\ge t(H,W)$ for every graphon $W$, we write $G\ge H$. Similarly, if $G$ and $H$ are vertex-colored graphs and $t(G,(W,f))\ge t(H,(W,f))$ for every vertex-colored graphon $(W,f)$, we again write $G\ge H$.

We work with the concept of adjoints introduced by Lov\'asz and Szegedy~\cite{2011:LovaszSzegedy:FinitelyForcibleGraphons} in their study of finite forcibility of graphons.
For this purpose, we consider functions $R$ that map a graphon $W$ to a measurable family of vertex-colored graphons, written $(R(W)_x,f_x)$, with $x$ sampled from the indexing measure space.

\begin{definition}
\label{def:adjoint}
    Let $L$ be a function that maps vertex-colored graphs to a formal linear combination of graphs.
    Let $R$ be a function that maps a graphon $W$ to a measurable family of vertex-colored graphons $(R(W)_x,f_x)$ indexed by some measure space $(Z,\mu)$ with $\mu(Z)\le 1$ and $x\in Z$. The space $(Z,\mu)$ depends on $W$. We say that $(L,R)$ is an adjunction if for all graphs $H$ and graphons $W$
    \begin{equation*}
        \int_Z t(H,(R(W)_x,f_x)) \,\mathrm d\mu(x) = t(L(H),W).
    \end{equation*}
    We call $L$ a left adjoint if there exists an $R$ such that $(L,R)$ is an adjunction and analogously $R$ a right adjoint.
\end{definition}

Lov\'asz and Szegedy~\cite[Definition~3.7]{2011:LovaszSzegedy:FinitelyForcibleGraphons} define adjoints for graphon operators through the analogous identity on formal linear combinations of uncolored graphs. Our formulation allows $R$ to take values in vertex-colored graphons.

The connection of adjoints to Sidorenko graphs is given by the following proposition.

\begin{proposition}
    \label{prop:preserves-Sidorenko}
    Let $(L,R)$ be an adjunction and $H$ a Sidorenko graph without isolated vertices. If $L(K_2)$ is Sidorenko and
    $e_H \cdot e_{L(K_2)}\le e_{L(I_{2e_H-v_H}\sqcup H)}$ then $L(I_{2e_H-v_H}\sqcup H)$ is Sidorenko.
\end{proposition}

\begin{proof}
    Without loss of generality assume that the measure $\mu$ associated to $W$ has non-zero total mass. Otherwise, $t(L(H),W)=0=t(K_2,W)$. In the remaining case, this follows directly from
\begin{align}
t(L(I_{2e_H-v_H}\sqcup H),W)
  &= \int_Z t(I_{2e_H-v_H}\sqcup H,(R(W)_x,f_x)) \,\mathrm d\mu(x)
     \notag\\
  &\ge \int_Z
     \left(t(K_2,(R(W)_x,f_x))\right)^{e_H}
     \,\mathrm d\mu(x)
     \label{eq:chain_adjoint} \\
  &\ge \mu(Z)^{1-e_H}
     \left(\int_Z t(K_2,(R(W)_x,f_x))\,\mathrm d\mu(x)\right)^{e_H}
     \notag\\
  &\ge t(L(K_2),W)^{e_H}
     \ge t(K_2,W)^{e_{L(I_{2e_H-v_H}\sqcup H)}} .
     \notag
\end{align}
    Here the equality follows from the adjunction, the first inequality follows
    from the colored Sidorenko inequality, 
    the second follows by normalizing the measure $\mu$ to have measure $1$ and Jensen's inequality, and the third inequality follows since $\mu(Z)\le 1$ and by the adjunction formula.
    The final two inequalities follow from the Sidorenko property of $L(K_2)$ and the edge count hypothesis respectively. 
\end{proof}

\begin{remark}
\label{rem:preserves-forcing}
The same chain gives a criterion for forcing.
Suppose that $H$ is
forcing and that the right adjoint has the following reconstruction
property: For every graphon $W$, if its right adjoint indexing space has $\mu(Z)>0$, the main-color edge density $t(K_2,(R(W)_x,f_x))$ is a fixed positive constant for $\mu$-almost every $x\in Z$, and $R(W)_x$ is constant almost everywhere on $\{f_x>0\}^2$ for $\mu$-almost every $x\in Z$, then $W$ is constant almost everywhere. Then $L(I_{2e_H-v_H}\sqcup H)$ is forcing.

Indeed, suppose $p=t(K_2,W)\in (0,1)$ and assume that 
\[ t(L(I_{2e_H-v_H}\sqcup H),W) = t(K_2,W)^{e_{L(I_{2e_H-v_H}\sqcup H)}}\]
Then, equality in \cref{eq:chain_adjoint} implies equality in the colored Sidorenko inequality almost everywhere.
Since $H$ is forcing, we have $e_H>1$, so equality in Jensen's inequality also implies that $t(K_2,(R(W)_x,f_x))$ is a fixed positive constant for $\mu$-almost every $x$.
Hence, $R(W)_x$ is almost everywhere constant on $\{f_x>0\}^2$ for $\mu$-almost every $x$. As $x$ varies $t(K_2,(R(W)_x,f_x))$ is constant because of the Jensen step. By the reconstruction property, this implies that $W$ is constant almost everywhere, as required.
\end{remark}

Motivated by \Cref{prop:preserves-Sidorenko}, our goal is now to construct a general class of adjunctions on graphs.
Let $a,b,c\in\NN$ be natural numbers (possibly zero), and let $F$ be a graph on vertex set $[2]\times [a]\sqcup [b]\sqcup [c]$ that is symmetric with respect to the two copies of $[a]$.
For a graph $H$ (with all vertices in the main color), define $L_{F,a,b,c}(H)$ as the graph on vertex set $[a]\times V_H\sqcup [b]\times E_H\sqcup [c]$ obtained by replacing every edge of $H$ with a copy of $F$, where the copies of $[a]$ are identified over the vertices of $H$ and all copies of $[c]$ are identified; the case $a=1$ and $c=0$ appears in~\cite[Example~3.13]{2011:LovaszSzegedy:FinitelyForcibleGraphons}.
For every $u\in V_H$, including isolated vertices, the construction also places a copy of $F[\{1\}\times[a]\sqcup[c]]$ on $[a]\times\{u\}\sqcup[c]$.
We think of $F[\{1\} \times [a]]$ as placed over each vertex of $H$, $F[\{1,2\} \times [a] \sqcup [b]]$ as placed over each edge of $H$, and $F[[c]]$ as the constant part connected to the rest.

For example, $L_{K_{m,m},m,0,0}(H)$ is the $m$-fold blow-up $H[m]$, $L_{C_4,2,0,0}(H)$ is the box product $H\Box K_2$, $L_{K_{2,1},1,1,0}(H)$ is the subdivision of $H$ (inserting a vertex on each edge), and $L_{K_3,1,0,1}(H)$ adds one new vertex connected to every vertex of $H$.
The operator $L_{K_3,1,0,1}$ is shown in \cref{fig:vertex-colored-graphon}; in this case, the corresponding right adjoint $R$ maps $W \equiv p$ to the graphon that is $p$ on $[0,p]^2$ and zero elsewhere.
There is some ambiguity when the vertices of $F$ are not labeled, but we will use this notation when the labeling is clear from context.

\begin{theorem}\label{thm:adjoint_existence}
    Let $a,b,c \in \NN$, and let $F$ be a graph on $[2]\times [a]\sqcup [b]\sqcup [c]$ that is symmetric with respect to the two copies of $[a]$.
    Then $L(H)=L_{F,a,b,c}(H)$ defines a left adjoint $L$. In particular, when $F[[c]]$ is an independent set, then for every graphon $W$ its space $(Z,\mu)$ has $\mu(Z)=1$.
\end{theorem}

\begin{figure}
\centering
    \begin{tikzpicture}[font=\small]

\def\NodeSize{1.5mm}        
\def\BigNodeSize{3mm}     
\def\CfourRadius{5mm}     
\def\LoopSize{3mm}        

\tikzset{
  vertex/.style={
    circle, draw, line width=0.5pt,
    minimum size=\NodeSize, inner sep=0pt
  },
  bigvertex/.style={
    circle, draw, line width=0.5pt,
    minimum size=\BigNodeSize, inner sep=0pt
  }
}


\node[bigvertex, fill=purple!50] (Ltop) at (0,0) {};
\node[bigvertex, gray!40, fill=white]   (Lbot) at (0,-1.2) {};

\node[anchor=east] at ($(Ltop.west)+(-2pt,0)$) {$p$};
\node[anchor=east] at ($(Lbot.west)+(-2pt,0)$) {$1-p$};

\draw (Ltop) edge[loop right, min distance=\LoopSize]
      node[midway, right] {$p$} (Ltop);

\draw (Lbot) edge[loop right, min distance=\LoopSize, gray!40]
      node[midway, right, black] {$0$} (Lbot);

\draw[gray!30, line width=0.7pt]
      (Ltop) -- node[midway, right, black] {$0$} (Lbot);

\node[vertex, fill=gray!40, scale=0.7mm] (R) at (3.0,-0.6) {};

\draw (R) edge[loop right, min distance=\LoopSize]
      node[midway, right] {$p$} (R);


\coordinate (MidLeft)  at ($(Ltop)!0.5!(Lbot)$);
\coordinate (C4center) at ($(MidLeft)+(0,-2.0)$);

\node[vertex, fill=purple!50] (A1) at ($(C4center)+(0,\CfourRadius)$) {};
\node[vertex, fill=purple!50] (A2) at ($(C4center)+(\CfourRadius,0)$) {};
\node[vertex, fill=purple!50] (A3) at ($(C4center)+(0,-\CfourRadius)$) {};
\node[vertex, fill=purple!50] (A4) at ($(C4center)+(-\CfourRadius,0)$) {};
\draw (A1)--(A2)--(A3)--(A4)--(A1);

\coordinate (C4centerB) at ($(C4center)+(3.0,0)$);

\node[vertex, fill=black] (B1) at ($(C4centerB)+(0,\CfourRadius)$) {};
\node[vertex, fill=black] (B2) at ($(C4centerB)+(\CfourRadius,0)$) {};
\node[vertex, fill=black] (B3) at ($(C4centerB)+(0,-\CfourRadius)$) {};
\node[vertex, fill=black] (B4) at ($(C4centerB)+(-\CfourRadius,0)$) {};
\draw (B1)--(B2)--(B3)--(B4)--(B1);

\node[vertex, fill=black]
      (Apex) at ($(C4centerB)$) {};

\draw (Apex)--(B1);
\draw (Apex)--(B2);
\draw (Apex)--(B3);
\draw (Apex)--(B4);

\end{tikzpicture}
\caption{Counting copies of the vertex-colored four-cycle $C_4$ in the left graphon gives the same value $p^8$ as counting copies of the four-wheel graph $L_{K_3,1,0,1}(C_4)$ in the right (Erd\H{o}s--R\'enyi) graphon. Note the (inevitable) $0$-weighted edges in the left graphon for this left adjoint, due to $c\ne 0$.}
    \label{fig:vertex-colored-graphon}
\end{figure}

Lov\'asz and Szegedy's Example~3.13 is exactly the case $a=1$ and $c=0$ of this construction, while their Examples~3.11 and~3.12 give the tensor-power and subdivision adjoints, respectively~\cite[Examples~3.11--3.13]{2011:LovaszSzegedy:FinitelyForcibleGraphons}.
For $c=0$, \cref{thm:adjoint_existence} unifies these examples and extends the general gadget substitution of Example~3.13 from $a=1$ to arbitrary $a$, using a vertex-colored right adjoint.
Thus, where Lov\'asz and Szegedy give a definition and several explicit examples, \cref{thm:adjoint_existence} provides a single construction for the entire class $L_{F,a,b,0}$, including the blow-up, Cartesian-product, and strong-product operators used here.
Whether all adjunctions on simple graphs arise as combinations of these operators remains open.
\Cref{prop:preserves-Sidorenko} gives the following corollary.

\begin{corollary}
    \label{cor:preserve-Sidorenko}
    If $F[(\{1\}\times[a])\sqcup[c]]$ is edgeless and $F$ is Sidorenko, then the left adjoint $L(H)= L_{F,a,b,c}(H)$ preserves the Sidorenko property.
\end{corollary}

Since $L_{K_{m,m},m,0,0}(H)$ is the $m$-fold blow-up and $L_{K_{m,m},m,0,0}(K_2)=K_{m,m}$ is Sidorenko, the corollary recovers the Sidorenko inequality used in the first assertion of \cref{thm:blowup}; its forcing conclusion follows from the direct equality argument above.
For the step graphon $W_H$ of a graph $H$, the graphon $R(W_H)$ realizes the auxiliary graph $H'$ from the direct blow-up proof above; we return to this comparison in \Cref{sec:concluding}.

\begin{proof}[Proof of \cref{thm:adjoint_existence}]
Fix a graphon $W$ represented on $[0,1]$.
We need to find a family of graphons $R(W)_x$  on spaces $\Omega_x$ that satisfy the adjunction rule. Recall the definition of $L=L_{F,a,b,c}$. For a graph $H$, the value $t(L(H),W)$ is exactly
\begin{equation}\label{eq:def_of_adjunction_for_rewrite}
    t(L(H),W)
    = \int_{[0,1]^{[a]\times V_H}\times [0,1]^{[b]\times E_H}\times [0,1]^{[c]}} \prod_{(u,v)\in E_{L(H)}} W(y_u,y_v) \, \mathrm d \mathcal L(y)
\end{equation}
and we want this to equal $\int_Z t(H,(R(W)_x,f_x))\,\mathrm d\mu(x)$, so an integral over $H$ instead of $L(H)$.
We define $R(W)_x$ by regrouping the variables in \cref{eq:def_of_adjunction_for_rewrite} according to the vertices and edges of $H$.

For a graph $G$ define on $[0,1]^{V_G}\times [0,1]^{\binom{V_G}{2}}$ the indicator $\mathbbm 1_G^W$ as the function
\begin{equation*}
    \mathbbm 1_G^W(y,\omega) 
    = \prod_{uv\in E_G}
     \mathbbm 1_{\omega_{uv}\le W(y_{u},y_{v})} \, .
\end{equation*}
Integrating first over the dummy variables $\omega$ gives
\[
    \int_{[0,1]^{\binom{V_G}{2}}} \mathbbm 1_G^W(y,\omega)\,\mathrm d\mathcal L(\omega) = \prod_{uv\in E_G}W(y_u,y_v).
\]
Integrating also over $y$ therefore gives
\[
    \int_{[0,1]^{V_G}\times[0,1]^{\binom{V_G}{2}}} \mathbbm 1_G^W(y,\omega)\,\mathrm d\mathcal L(y,\omega) = t(G,W).
\]
With $\tilde F=F[\{1\}\times [a] \sqcup [c]]$ and $\bar F=F[[c]]$, idempotence of the indicators gives
\[
\mathbbm 1^W_{L(H)}
= \mathbbm 1^W_{\bar F}
\prod_{v\in V_H}\mathbbm 1^W_{\tilde F}
\prod_{e\in E_H}\mathbbm 1^W_F.
\]
The first factor contains the variables shared globally, the second contains the variables associated with each vertex of $H$, and the third contains those associated with each edge of $H$.
The purpose of the dummy variables $\omega$ is that our target graphon $W$ might have edge weights in $(0,1)$; to distinguish whether the random graph generated by $W$ at a given edge is realized by an edge or a non-edge, we use the dummy variable together with an indicator function whose threshold is the edge weight of $W$ there.

We let $\Omega_{0}=([0,1]^{[a]}\times [0,1]^{[a] \times [c] \sqcup \binom{[a]}{2}},\mathcal L)$ with $(x,\xi) \in [0,1]^{[c]}\times [0,1]^{\binom{[c]}{2}}$ indexed by the subprobability measure $\mathrm d\mu(x,\xi) = \mathbbm 1_{\bar F}^W(x,\xi)\,\mathrm d\mathcal L(x,\xi)$.
Then define the vertex color
\begin{equation}
    \label{eq:tight_adjoint_vertex_color}
    f_{x,\xi}(y,\omega) = \mathbbm 1^W_{\tilde F}((y,x),(\omega,\xi)),
\end{equation}
and the graphon $R(W)_{x,\xi}:\Omega_{0}^2 \to [0,1]$
\begin{equation}\label{eq:representation_of_right_adjoint}
    R(W)_{x,\xi}((y,\omega),(z,\rho)) 
    = \int_{(r,\psi)\in [0,1]^{[b]}\times [0,1]^{E_{F'}}} \mathbbm 1^W_F((y,z,r,x),(\omega,\rho,\psi,\xi)) 
    \,\mathrm d \mathcal L(r,\psi),
\end{equation}
where $F'$ is the subgraph of $F$ with edge-set $E_F \setminus \left( \binom{(\{1\}\times [a])\cup [c]}{2} \cup \binom{(\{2\}\times [a])\cup [c]}{2} \right)$.
The symmetry of $F$ makes $R(W)_{x,\xi}$ symmetric.
Informally, $R(W)_{x,\xi}$ is the graphon obtained from $\mathbbm 1^W_F$ by integrating out all variables associated with the $b$-vertices and with the edges between the two $a$-parts.

This vertex-colored graphon has the required property.
To verify this, we evaluate
\begin{equation}\label{eq:calculation_of_adjoint}
\int_{[0,1]^{[c]} \times [0,1]^{\binom{[c]}{2}}} \int_{\Omega_0^{V_H}}\prod_{v\in V_H} f_{x,\xi}(y_v,\omega_v)\prod_{uv\in E_H} R(W)_{x,\xi}((y_u,\omega_u),(y_v,\omega_v)) \, \mathrm d \mathcal L(y,\omega) \,\mathrm d\mu(x,\xi).
\end{equation}
Shared dummy variables in $\mathbbm 1^W_{F}$, $f_{x,\xi}$, and $R(W)_{x,\xi}$ are integrated over only once.
Plugging \cref{eq:representation_of_right_adjoint} into \cref{eq:calculation_of_adjoint}, we find all vertices of $L(H)$ among $y$, $r$ (one $r$ per $uv \in E_H$), and $x$, which we collect as $z \in [0,1]^{[a]\times V_H}\times [0,1]^{[b]\times E_H}\times [0,1]^{[c]}$.
The terms $W(z_u,z_v)$ for the edges $uv\in E_{L(H)}$ can be identified as follows.
For $v \in V_H$, the $\omega_v$ from $\int_{\Omega_0^{V_H}}$ account for the edges within $[a] \times \{v\}$ and between $[a] \times \{v\}$ and $[c]$, i.e.~the edges \emph{over the vertices of $H$}.
All edges touching $[b] \times \{uv\}$, as well as those between $\{1\} \times [a] \times \{u\}$ and $\{2\} \times [a] \times \{v\}$ for $uv \in E_H$, are covered by the respective $\psi$ from $R(W)_{x,\xi}((y_u,\omega_u),(y_v,\omega_v))$, i.e.~the edges \emph{over the edges of $H$}.
Finally, the edges within $[c]$ are accounted for by $\xi$, i.e.~the edges \emph{within the constant part}.
Hence we obtain
\[ \int_{[0,1]^{[a]\times V_H}\times [0,1]^{[b]\times E_H}\times [0,1]^{[c]}} \prod_{uv\in E_{L(H)}} W(z_u,z_v)\, \mathrm d \mathcal L(z) = t(L(H),W) \, . \qedhere \]
\end{proof}

\section[Proof of the triple box-product proposition]{Proof of \Cref{thm:box-product}}
\label{sec:proof-box-product}

Applying the box product with $F$ to the uncolored Sidorenko inequality $G \ge K_2^{e_G}$ produces the edge-unbalanced inequality $G\Box F \ge (K_2\Box F)^{e_G}$: the two sides have $v_G e_F + v_F e_G$ and $e_G(v_F + 2e_F)$ edges, respectively. We correct this imbalance using the vertex-colored inequality $I_{2e_G-v_G}\sqcup G \ge K_2^{e_G}$. Applying a left adjoint with $a\ne 0$ introduces the extra factor $L(I_{2e_G-v_G})$; the hypotheses of \cref{thm:box-product} allow us to cancel this factor.
 
\begin{proof}[Proof of \cref{thm:box-product}]
    We first assume that $G$ and $F$ have no isolated vertices; the general case is reduced to this one at the end of the proof.
    Throughout, $W$ is an arbitrary non-zero graphon.
    Then every finite bipartite graph has positive density in $W$.
    Before beginning the actual proof, we note that any inequality $A\ge B$ between uncolored graphs can be upgraded to an inequality between vertex-colored graphs. If $B$ has at least as many vertices as $A$, then the new inequality is $A\sqcup I_{v_B-v_A} \ge B$ (here with all vertices in the main color).
    If $\lVert f\rVert_1=0$, the claimed inequality is immediate. Otherwise, for any vertex-colored graphon $(W,f)$,
    \[
    t(A\sqcup I_{v_B-v_A}, (W,f))
    =
    t(A, V) \cdot \lVert f\rVert_1^{v_B}
    \ge
    t(B,(W,f)),
    \]
    where $V$ is the graphon on $([0,1], f/\lVert f\rVert_1 \, \mathcal L)$ with function value $V(x,y) = W(x,y)$.
    Hence $A\sqcup I_{v_B-v_A} \ge B$ is valid for all vertex-colored graphons.

    We know that by \cref{thm:adjoint_existence} the operator $L$ that maps a graph $A$ to $A\Box Q_n$ or $A\Box F$ has a right adjoint.
    Also, all right adjoints that we are using in this proof have no variation in $x$, because $c=0$, so we drop it from the notation.
    
    Since $G$ is Sidorenko, we can upgrade the uncolored inequality $G \ge K_2^{e_G}$ to an inequality with vertices colored with the main color. Hence, for every right adjoint $R$ and graphon $W$,
    \begin{equation*}
    t(G\sqcup I_{2e_G-v_G},(R(W),f)) \ge t(K_2,(R(W),f))^{e_G} \, .
    \end{equation*}
    Choosing $R$ to be the right adjoint to the left adjoint $L$ that is the box product with $Q_n$, we then get
    \begin{equation*}
        t((G\sqcup I_{2e_G-v_G})\Box Q_{n},W)=t(G\Box Q_{n},W)t(Q_n,W)^{
        2e_G-v_G} \ge t(Q_{n+1},W)^{e_G}.
    \end{equation*}
    
    Denote by $p_n=2^n$ and $q_n=n2^{n-1}$ the number of vertices and edges of $Q_n$, respectively. 
    Since $Q_{n+1}$ contains $Q_{n}$ we have by domination $Q_{n+1}^{q_n} \ge Q_n^{q_{n+1}}$. 
    Then,
    \begin{equation}\label{eq:main_box_product_relation}
        t(G\Box Q_{n},W)^{q_{n+1}}
        \ge 
        t(Q_{n+1},W)^{q_{n+1}e_G-q_n(2e_G-v_G)}
        =
        t(Q_{n+1},W)^{p_ne_G+q_nv_G} \, .
    \end{equation}
    Let $R'$ be the right adjoint for box products with $F$.
    Regard both sides of the graph inequality underlying \Cref{eq:main_box_product_relation}, after adjoining the required $I_z$ on the left, as main-colored graphs and evaluate them on $(R'(W),f)$.
    Applying the adjoint relation yields
    \begin{equation*}
        t(G\Box Q_{n},(R'(W),f))^{q_{n+1}} t(I_{z},(R'(W),f))
        \ge
        t(Q_{n+1},(R'(W),f))^{p_ne_G+q_nv_G} \, ,
    \end{equation*}
    where $z = p_{n+1}(p_ne_G+q_nv_G) - q_{n+1}v_Gp_n = p_n^2(2e_G-v_G)$.
    Then,
\begin{equation}\label{eq:second_main_box_product_relation}
        t(G\Box F \Box Q_{n},W)^{q_{n+1}}
        t(F,W)^{p_{n}^2(2e_G-v_G)}
        \ge
        t(F \Box Q_{n+1},W)^{e_Gp_n + q_nv_G} \, .
    \end{equation}

    We need to cancel the $t(F,W)$ on the left hand side of this equation.
    To this end, we lower bound the right hand side. Applying \cref{eq:main_box_product_relation} with $G=F$ and with $n$ replaced by $n+1$ gives
    \begin{equation}\label{eq:hypercube_box_H_is_forcing}
        t(F \Box Q_{n+1},W)^{q_{n+2}}
        \ge
        t(Q_{n+2},W)^{e_Fp_{n+1}+q_{n+1}v_F} \, .
    \end{equation}
    Now invoke the assumption about domination $t(Q_{n+2},W)^{e_F} \ge t(F,W)^{q_{n+2}}$ and further develop \cref{eq:hypercube_box_H_is_forcing} to
    \begin{equation}
    \label{eq:hypercube_box_H_is_forcing_2}
        t(F \Box Q_{n+1},W)^{e_F}
        \ge 
        t(F,W)^{e_Fp_{n+1}+q_{n+1}v_F} \, .
    \end{equation}

    From \cref{eq:second_main_box_product_relation} and \cref{eq:hypercube_box_H_is_forcing_2} we now get
    \begin{equation}
        t(G\Box F \Box Q_{n},W)^{q_{n+1}e_F}
        \ge
        t(F,W)^{(e_Gp_n + q_nv_G)(e_Fp_{n+1}+q_{n+1}v_F)-e_Fp_{n}^2(2e_G-v_G)} \, ,
    \end{equation}
    which, after some calculation using $p_{n+1}=2p_n$, $q_{n+1}=2q_n+p_n$, and $e_{G\Box F \Box Q_{n}}=(e_Gv_F+v_Ge_F) p_n +v_Gv_F q_n$, gives
    \begin{equation}\label{eq:final_inequality_for_Box}
    t(G\Box F \Box Q_{n},W) \ge t(F,W)^{e_{G\Box F \Box Q_{n}}/e_F} \ge t(K_2,W)^{e_{G\Box F \Box Q_{n}}} \, .
    \end{equation}
    The last inequality uses the Sidorenko property of $F$.

    For forcing, suppose that $t(G\Box F \Box Q_{n},W) = t(K_2,W)^{e_{G\Box F \Box Q_{n}}}$.
    Put $p=t(K_2,W)$.
    If $p \in \{0,1\}$, then $W$ is already constant almost everywhere.
    We may therefore assume $p \in (0,1)$, and then by \Cref{eq:final_inequality_for_Box}, also $t(F,W)=t(K_2,W)^{e_F}$.
    Combining \Cref{eq:second_main_box_product_relation} with the Sidorenko bound for $F\Box Q_{n+1}$ gives $t(F\Box Q_{n+1},W)=t(K_2,W)^{e_{F\Box Q_{n+1}}}$. Substituting this equality into \Cref{eq:hypercube_box_H_is_forcing} gives $t(Q_{n+2},W)\le t(K_2,W)^{q_{n+2}}$, while the Sidorenko property of $Q_{n+2}$ gives the reverse inequality. Thus equality holds for the forcing graph $Q_{n+2}$~\cite{2010:Hatami:GraphNormsSidorenko,2010:ConlonEtAl:ApproximateSidorenko}, so $W$ is constant.

    If $G$ or $F$ has isolated vertices, delete them to obtain $G^\circ$ and $F^\circ$, which satisfy the same hypotheses. Then $G\Box F\Box Q_n$ is the disjoint union of $G^\circ\Box F^\circ\Box Q_n$ and copies of $G^\circ\Box Q_n$, $F^\circ\Box Q_n$, and $Q_n$. The first component is Sidorenko and forcing by the preceding argument, while all remaining components are Sidorenko by iterating the Cartesian-product result of Kim, Lee, and Lee~\cite{2016:KimEtAl:TwoApproachesSidorenko}. A disjoint union of Sidorenko graphs with a forcing component is forcing, since equality in the product of the componentwise lower bounds forces equality in that component.
\end{proof}

\section{Adjoints preserving forcing families}
\label{sec:additional_results}

We now return to forcing families and work towards identifying adjoints that preserve the forcing property of a family similarly to \Cref{rem:preserves-forcing} with a single graph.
For this we encode the forcing property as a zero set, similarly to Lov\'asz~\cite{2012:Lovasz:LargeNetworksGraphLimits}, who called these sets kernel varieties.

\begin{definition}
\label{def:zero_set}
   Let $I$ be a collection of formal linear combinations of uncolored graphs or of vertex-colored graphs whose vertices all have the main color. We denote by $\mathcal Z(I)$ the set of uncolored or vertex-colored graphons, respectively, on which every $g\in I$ has density zero.
\end{definition}

For example, when $\mathcal G$ is a family of vertex-uncolored graphs, the notion of being a $p$-forcing family for $p\in (0,1)$ is equivalent to requiring that the zero set $\mathcal Z(\{G-p^{e_G}\mid G\in \mathcal G\})$ consists of graphons equal almost everywhere to the constant-$p$ graphon.
Moreover, when $\mathcal G$ is forcing, then $\mathcal Z(\{ G - p^{e_G}I_{v_G} \mid G\in \mathcal G\})$, with all vertices in $G$ and $I_{v_G}$ in the main vertex color, consists of all vertex-colored graphons $(W,f)$ such that either $\lVert f\rVert_1=0$ or, with $\mathrm d\mu=f\,\mathrm d\mathcal L/\lVert f\rVert_1$, $W=p$ almost everywhere on $A_f^2$.
In this way, we can transport a vertex-uncolored forcing property to a vertex-colored forcing property.

We seek conditions under which an adjoint $L$ transports zero-set inclusions, so that $\mathcal Z(I) \subseteq \mathcal Z(J)$ implies
$\mathcal Z(L(I)) \subseteq \mathcal Z(L(J))$ for families of formal linear combinations $I,J$.
Thus, if $\mathcal Z(J)$ and $\mathcal Z(L(J))$ consist only of the relevant constant graphons, the same conclusion for $\mathcal Z(I)$ implies it for $\mathcal Z(L(I))$.
We have the following lemma.

\begin{lemma}
    \label{lem:forcing_families}
    Let $L=L_{F,a,b,c}$ be a left adjoint for graphs. 
    \begin{enumerate}[label=(\roman*)]
        \item Let $c=0$, and let $I,J$ be sets of formal linear combinations of vertex-colored graphs whose vertices all have the main color. If $\mathcal Z(I) \subseteq \mathcal Z(J)$, then $\mathcal Z(L(I)) \subseteq \mathcal Z(L(J))$. By symmetry, if the first subset relation is an equality, then so is the second.
        \item Let $c\ne 0$, and let $I,J$ be sets of formal linear combinations of vertex-colored graphs whose vertices all have the main color. Suppose that every $g\in I$ is nonnegative, meaning that $t(g,(W,f))\ge0$ for every vertex-colored graphon $(W,f)$. If $\mathcal Z(I) \subseteq \mathcal Z(J)$, then $\mathcal Z(L(I)) \subseteq \mathcal Z(L(J))$.
    \end{enumerate}
\end{lemma}

For the triangle-substitution operator $L_{K_3,1,1,0}$, the corresponding argument is implicit in Reiher and Schacht~\cite[Theorem~4.2 and the proof of Corollary~1.4]{2019:ReiherSchacht:ForcingQuasirandomness}.
If the induced subgraph of $F$ on $(\{1\}\times[a])\sqcup[c]$ is empty, then $L$ makes no difference between vertex-colored and uncolored graphs. Applying \cref{lem:forcing_families} proves \cref{thm:subdivisions}:
\begin{proof}[Proof of \cref{thm:subdivisions}]
    The tuple $(C_{4},K_2)$ is a forcing pair. Let $\ell = k+1$ and denote by $P_\ell$ the path with $\ell$ edges.
    We first argue that $(C_{4\ell},P_{\ell})$ is again a forcing pair and to this end, we assume that $W$ is a graphon so that $t(C_{4\ell},W) = t(P_{\ell},W)^{4}$. 
    Write $L=L_{P_{\ell},1,k,0}$ for the operator, where the endpoints of $P_{\ell}$ are the two vertices in $[2] \times [a]$, and denote by $W^{\circ \ell}$ the usual operator product.
    This is an adjunction as in \Cref{thm:adjoint_existence} (also see the discussion after the theorem and~\cite[Example~3.12]{2011:LovaszSzegedy:FinitelyForcibleGraphons}) and we have 
    \begin{equation*}
       t(C_4,W^{\circ \ell}) = t(L(C_{4}),W) =
        t(C_{4\ell},W) = t(P_{\ell},W)^{4} = t(L(K_2),W)^{4}
        =
        t(K_2,W^{\circ \ell})^4
    \end{equation*}
    Set $r=t(K_2,W^{\circ\ell})$.
    If $r\in(0,1)$, the displayed equality and the forcing property of $(C_4,K_2)$ imply that $W^{\circ\ell}$ is the constant-$r$ graphon; the same conclusion follows directly from $0\le W^{\circ\ell}\le1$ when $r\in\{0,1\}$.

    Regard $W$ as the kernel of the compact self-adjoint integral operator $T_W$ on $L^2([0,1])$.
    Then $W^{\circ\ell}$ is the kernel of $T_W^\ell$.
    If $T_W^\ell=0$, the spectral theorem gives $T_W=0$.
    Otherwise, $T_W^\ell$ has rank one and its non-zero eigenspace is spanned by the constant function.
    Since the eigenvalues of $T_W^\ell$ are the $\ell$-th powers of the eigenvalues of $T_W$, it follows that $T_W$ also has rank one and its non-zero eigenfunction is constant.
    In either case, $W$ is constant almost everywhere.
    
    Now, let $\mathcal F$ be a forcing family.
    For any $p\in(0,1)$,
    \begin{equation*}
        \mathcal Z(\{F - p^{e_F} \mid F\in\mathcal F \})
        = 
        \mathcal Z(\{C_{4}-p^{4},K_2-p\}).
    \end{equation*}
    By \cref{lem:forcing_families}, 
    \begin{equation*}
        \mathcal Z(\{L(F) - p^{e_F} \mid F\in\mathcal F \})
        = 
        \mathcal Z(\{C_{4\ell}-p^{4},P_{\ell}-p\}).
    \end{equation*}
    Set $q=p^{1/\ell}$.
    Since $p$ ranges over $(0,1)$, so does $q$, and $e_{L(F)}=\ell e_F$ for every $F\in\mathcal F$.
    The preceding equality and the forcing property of $(C_{4\ell},P_\ell)$ therefore show that
    \begin{equation*}
        \{ L(F) \mid F\in\mathcal F \}
    \end{equation*}
    is $q$-forcing for every $q\in(0,1)$ and therefore forcing.
\end{proof}

\begin{proof}[Proof of \cref{lem:forcing_families}]
    For the first item, let $W\in \mathcal Z(L(I))$.
    Since $c=0$, the right adjoint is deterministic, so we drop the subscript $x$ from the notation. By the adjoint relation, for every $g\in I$
    \begin{equation*}
        0=t(L(g),W)
        =
        t(g,(R(W),f)) \, .
    \end{equation*}
    Hence $(R(W),f)\in \mathcal Z(I)\subseteq \mathcal Z(J)$, and the adjoint relation gives $0 = t(L(h),W)$ for all $h\in J$. This proves $W\in \mathcal Z(L(J))$.

    For the second claim, let $W \in \mathcal Z(L(I))$.
    Choose a countable subset $I_0 \subseteq I$ such that $\operatorname{span}(I_0)=\operatorname{span}(I)$; this is possible since the vector space of formal linear combinations of finite vertex-colored graphs has a countable basis. For every $g\in I_0$, the adjunction and the non-negativity of $g$ give
    \[
    0=t(L(g),W) =\int_Z t(g,(R(W)_x,f_x))\, \mathrm d\mu(x),
    \]
    and hence $t(g,(R(W)_x,f_x))=0$ for almost every $x\in Z$.
    Intersecting these countably many full-measure sets shows that $(R(W)_x,f_x)\in\mathcal Z(I_0)=\mathcal Z(I)$ for almost every $x\in Z$.
    Hence, almost surely,$(R(W)_x,f_x)\in\mathcal Z(J)$, and the claim follows as before.
\end{proof}

To apply \cref{lem:forcing_families}, we need a forcing family whose image under $L$ can be controlled.
The next lemma is motivated by H\`an, Person, and Schacht~\cite{2011:HanEtAl:NoteForcingPairs}, who proved that every graph $H$ belongs to a forcing pair $\{M(H),H\}$.

\begin{lemma}
    \label{thm:forcing_pairs}
    Let $L=L_{F,a,0,0}$, with $\tilde F = F[\{1\}\times [a]]$ bipartite and with at least one edge between $\{1\}\times [a]$ and $\{2\}\times [a]$ in $F$.
    Let $G$ be a forcing graph. Then for any $p\in(0,1)$,
    \begin{equation}\label{eq:zero_set_from_forcing}
        \mathcal Z(L(\{G - p^{e_G}I_{v_G}, K_2 - pI_2\}))
        \subseteq\{\alpha \mathbbm 1_A\mid \alpha\in [0,1],\, A\subseteq[0,1]\text{ measurable}\} \, .
    \end{equation}
\end{lemma}
Here $\alpha\mathbbm 1_A$ denotes the graphon that equals $\alpha$ on $A\times A$ and zero elsewhere.
To prove that a particular image family is forcing, it remains to show that its density equations exclude graphons $\alpha\mathbbm 1_A$ with $\mathcal L(A)<1$.

\begin{proof}[Proof of second part of \cref{thm:blowup}, \cref{prop:Box_Product_Preserving}, and \cref{prop:Strong_Product_Preserving}]
    Let $H$ be bipartite.
    Fix $p \in (0,1)$, and take a forcing pair like $(K_2,C_4)$. Let $J=\{C_4 - p^{4}I_4, K_2 - pI_2\}$. 
    Applying \cref{thm:forcing_pairs} to the box product $L=L_{H\Box K_2,v_H,0,0}$ shows that the zero set of
    \begin{equation*}
       L(J) = \{ H\Box C_4 - p^{4} H^4, H\Box K_2 - pH^2 \}
    \end{equation*}
    is contained in the set of graphons of the form $\alpha\mathbbm 1_A$.
    
    Now let $\mathcal F$ be a forcing family. Any edgeless members of $\mathcal F$ may be removed, since their density equations hold for every graphon. Deleting isolated vertices from the remaining graphs does not change ordinary homomorphism densities, so we may assume that every $F\in\mathcal F$ has at least one edge and no isolated vertices. For either product, each deleted vertex contributes a disjoint copy of $H$, whose density is already fixed because $H$ belongs to the image family.

    Define
    $I=\{F - p^{e_F} I_{v_F}\mid F\in \mathcal F\}$.
    Then, $\mathcal Z(I)=\mathcal Z(J)$ and so by \Cref{lem:forcing_families}, the zero set of
    \begin{equation}\label{eq:zero_set_for_box_product}
        L(I) = \{(F\Box H) - p^{e_F} H^{v_F}\mid F \in \mathcal F\}
    \end{equation}
    is contained in the set of graphons of the form $\alpha\mathbbm 1_A$.
    Now consider 
    \begin{equation*}
        \{F\Box H\mid F \in \mathcal F\} \cup \{H\}.
    \end{equation*}
    If $W$ is a graphon and $q\in(0,1)$ satisfies
    \begin{equation*}
        t(F\Box H,W) = q^{e_{F\Box H}} \quad \text{and} \quad t(H,W)= q^{e_H}
    \end{equation*}
    then $W$ is in the zero set of \cref{eq:zero_set_for_box_product} with $p=q^{v_H}$.

    The two density equations force $\mathcal L(A)=1$.
    Here $\alpha >0$ and $\mathcal L(A)>0$ because $t(H,W)=q^{e_H}>0$.
    Indeed,
    \begin{equation*}
        t(F\Box H, \alpha\mathbbm 1_A) = \alpha^{e_{F\Box H}} \mathcal L(A)^{v_{F\Box H}}
        \quad \text{and} \quad
        t(H, \alpha\mathbbm 1_A) = \alpha^{e_H} \mathcal L(A)^{v_H}
    \end{equation*}
    Raising the first equation to the power $e_H$, raising the second to the power $e_{F\Box H}$, and dividing gives
    \begin{equation} 
    \label{eq:final_box_implies_constant}
        e_H v_{F\Box H} = e_{F\Box H} v_H
    \end{equation}
    if $\mathcal L(A)<1$. Substituting the vertex and edge counts of the Cartesian product then gives $0 = e_F v_H^2$, a contradiction.
    This shows the claim about box products.

    The claim about the strong product $L=L_{H \boxtimes K_2, v_H, 0,0}$ uses the same proof template.
    For fixed $q \in (0,1)$ the transported parameter is $p=q^{v_H+2e_H}$, because $e_{F \boxtimes H}=e_Fv_H+e_Hv_F+2e_Fe_H$.
    With this parameter, the target density equations place $W$ in the transported zero set.
    The equation analogous to \cref{eq:final_box_implies_constant} for the strong product is
    \begin{equation*}
        e_H v_{F\boxtimes H} = e_{F\boxtimes H} v_H
    \end{equation*}
    which, if $\mathcal L(A)<1$, would imply $0=e_F v_H(v_H+2e_H)$, impossible when $e_F>0$.
    
    Finally, fix $p\in(0,1)$, let
    \[
    L_m=L_{K_{m,m},m,0,0}, \qquad J_p=\{C_4-p^4I_4,K_2-pI_2\}.
    \]
    By \Cref{thm:forcing_pairs}, every graphon in $\mathcal Z(L_m(J_p))$ has the form $W=\alpha \mathbbm 1_A$.
    Write $\lambda=\mathcal L(A)$.
    The two equations defining this zero set give
    \[
    \alpha^{4m^2}\lambda^{4m}=p^4 \qquad\text{and}\qquad \alpha^{m^2}\lambda^{2m}=p.
    \]
    Since $p>0$, we have $\alpha,\lambda>0$, and comparison with the fourth power of the second equation yields $\lambda=1$.
    Hence $\alpha=p^{1/m^2}$.

    Now let $\mathcal F$ be a forcing family and put
    \[
    I_p=\{F-p^{e_F}I_{v_F}\mid F\in\mathcal F\}.
    \]
    Then $\mathcal Z(I_p)=\mathcal Z(J_p)$, so \Cref{lem:forcing_families} gives
    \[
    \mathcal Z(L_m(I_p))=\mathcal Z(L_m(J_p)).
    \]
    Since $L_m(F)=F[m]$, $L_m(I_{v_F})=I_{mv_F}$, and $e_{F[m]}=m^2e_F$, reparameterizing by $q=p^{1/m^2}$ shows that $\{F[m]\mid F\in\mathcal F\}$ is forcing.
\end{proof}

\begin{proof}[Proof of \cref{thm:forcing_pairs}]
     We first use the forcing equations and the adjunction to show that the right-adjoint graphon $R(W)$ is constant on $A_f^2$.
     We then use the explicit product representation of $R(W)$ to pull this constancy back to $W$, showing that $W$ is constant on its support.

     Fix a non-zero candidate graphon $W$ that is in the zero set of the left hand side of \cref{eq:zero_set_from_forcing}.
     Then, we have
     \begin{align*}
         0&=t(L(G - p^{e_G} I_{v_G}),W) = t(G - p^{e_G} I_{v_G},(R(W),f)) \\
         0&=t(L(K_2 - p I_{2}),W) = t(K_2 - p I_2,(R(W),f)) \, .
     \end{align*}
     Put
    \[
    m=\int_{\Omega_0}f(y)\,\mathrm d\mathcal L(y).
    \]
    We first note that $m>0$. Indeed, if $m=0$, then the adjunction gives
    \[
    0=t(I_1,(R(W),f))=t(L(I_1),W)=t(\tilde F,W).
    \]
    If $\tilde F$ is edgeless, this contradicts $t(\tilde F,W)=1$. Otherwise, since $\tilde F$ is bipartite and $W\ne 0$, we have $t(\tilde F,W)>0$, again a contradiction.

    Normalize the color measure by
    \[
    \mathrm d\nu=\frac{f \, \mathrm d\mathcal L}{m}
    \]
    and regard $R(W)$ as a graphon on this probability space.
    The two displayed forcing equations become
    \[
    t(G,R(W))=p^{e_G} \qquad\text{and}\qquad t(K_2,R(W))=p,
    \]
    where the densities are taken with respect to $\nu$.
    Since $G$ is forcing, $R(W)=p$ almost everywhere on $A_f^2$.
    Set $\beta=p>0$.
     
    Following the proof of \cref{thm:adjoint_existence} which simplifies for our $L=L_{F,a,0,0}$, the vertex-colored graphon $(R(W),f)$ has domain $\Omega_0 = [0,1]^{[a]}$.
     Moreover, by
     \cref{eq:tight_adjoint_vertex_color}
     \begin{equation}\label{eq:vertex_color_function_right_adjoint}
         f(y) = \prod_{uv\in E_{\tilde F}} W(y_u,y_v)
     \end{equation}
     and by \cref{eq:representation_of_right_adjoint}, 
     \begin{equation}\label{eq:right_adjoint_recall_for_forcing_pairs}
         R(W)(y,z) = \prod_{uv\in E_{F'}}
         W(y_u,z_v),
     \end{equation}
     where $F_i=F[\{i\}\times[a]]$ for $i=1,2$, $\tilde F\cong F_1\cong F_2$, and $F'=F\setminus (F_1\cup F_2)$.
     
    Define 
    \begin{align*}
        S &= \{(s,s')\in (0,1)^2\mid  W(s,s')>0\}.
        \\
        M &= \{z\in [0,1]^2\mid z\text{ is a common Lebesgue point of }\mathbbm 1_S\text{ and }W\}
    \end{align*}
    Then, $\mathcal L(M) = 1$.
    Furthermore, let 
    \begin{align*}
        \rho(s) &= \mathcal L(\{s'\in [0,1]\mid (s,s')\in S\})
        \\
        A &= \{s\in [0,1]\mid \rho(s)>0\}
    \end{align*} 
    
    Choose $(s_1^{(j)},s_2^{(j)})\in S$ independently for $j=1,2$, with respect to normalized Lebesgue measure on $S$.
    By symmetry, each coordinate of a normalized random point of $S$ has density $\rho(s)/\mathcal L(S)$.
    By Fubini's theorem, for almost every such choice, every ordered pair formed from two distinct coordinate variables belongs to $M$.
    Since $\tilde F$ is bipartite, it has a (not necessarily unique) bipartition $V_{\tilde F}= V_1 \sqcup V_2$.
    Then, define points $q_1,q_2\in \Omega_0$ by setting $q_1$ equal to $s_1^{(1)}$ on $V_1$ and $s_2^{(1)}$ on $V_2$ and, similarly, setting $q_2$ equal to $s_1^{(2)}$ on $V_1$ and $s_2^{(2)}$ on $V_2$.
    Let $B_\varepsilon=[-\varepsilon,\varepsilon]^{[a]}$, the closed $\ell_\infty$-ball of radius $\varepsilon$. By \cref{eq:vertex_color_function_right_adjoint},
    \begin{equation*}
        1-\mathbbm 1_{A_f}(y)
        \le \sum_{uv\in E_{\tilde F}}(1-\mathbbm 1_S(y_u,y_v)).
    \end{equation*}
    Averaging over $\Omega_0\cap(B_\varepsilon+q_j)$ shows that $q_j$ is a Lebesgue point of $\mathbbm 1_{A_f}$, because every summand tends to zero. Likewise, the inequality
    \begin{equation*}
        \left|\prod_e x_e-\prod_e y_e\right|\le \sum_e|x_e-y_e|,
        \qquad x_e,y_e\in[0,1],
    \end{equation*}
    applied to \cref{eq:right_adjoint_recall_for_forcing_pairs} shows that $(q_1,q_2)$ is a Lebesgue point of $R(W)$ after averaging over the product of the two balls.
    Therefore,
    \begin{gather}
    \label{eq:vertex_color_for_zero_set_comp}
        \lim_{\varepsilon\to 0}\frac{1}{\mathcal L(\Omega_0\cap (B_\varepsilon+q_j))} \int_{\Omega_0\cap (B_\varepsilon+q_j)} \mathbbm 1_{A_f}(y)\; \mathrm d\mathcal L(y)
        = 
        1.
    \end{gather}
    Hence, as $\varepsilon\to0$, the proportion of points in the ball around $q_j$ that belong to $A_f$ tends to $1$.
    Since $R(W)$ equals $\beta$ almost everywhere on $A_f^2$ and $(q_1,q_2)$ is a Lebesgue point of $R(W)$, it follows that
    \begin{gather}
    \begin{gathered}
    \label{eq:constant_edge_weight_for_forcing}
        \beta = \lim_{\varepsilon\to 0}
        \Big(\prod_{j=1}^2
        \frac{1}{\mathcal L(\Omega_0\cap (B_\varepsilon+q_j))} \Big)
        \int_{\prod_{j=1}^2(\Omega_0\cap (B_\varepsilon+q_j))}
        R(W)(y,z)\; \mathrm d\mathcal L(y,z)\\
        = \prod_{1 \le i,j \le 2}W(s_i^{(1)},s_j^{(2)})^{d_{i,j}}
    \end{gathered}
    \end{gather}
    for integers $d_{i,j}\ge 0$ that depend on the number of edges in $F'$ between $V_i\times \{1\}$ and $V_j\times \{2\}$.

    Orient the bipartitions of the two copies independently so that $d_{1,j}\ge1$ for some $j\in\{1,2\}$; this is possible by the assumption on $F$.
    In particular, $d=d_{1,1}+d_{1,2}>0$.
    The product in \Cref{eq:constant_edge_weight_for_forcing} is positive almost everywhere, and hence $W(s_1^{(1)},s_j^{(2)})>0$
    for almost every choice of the two points of $S$.
    Each of the two indicated coordinate projections has density $\rho(s)/\mathcal L(S)$, which is positive on $A$.
    Fubini's theorem therefore gives $W(x,y)>0$ for almost every $(x,y)\in A^2$, so $A^2\subseteq S$ almost everywhere.
    Conversely, by Fubini's theorem and the definition of $A$, we have $S\subseteq A^2$ almost everywhere.
    Thus, $S=A^2$ almost everywhere.
    
    It remains to show that $W$ is constant on $A\times A$.
    By Fubini's theorem, fixing all variables except $s_1^{(1)}$ shows that
    \begin{equation}\label{eq:final_path2_constant}
         W(s_1^{(1)},s^{(2)}_1)^{d_{1,1}}
         W(s_1^{(1)},s^{(2)}_2)^{d_{1,2}}
    \end{equation}
    is constant as a function of $s_1^{(1)}$ for almost every fixed $(s_1^{(2)},s_2^{(2)})\in A^2$. This constant is $>0$ because $\beta>0$.
    Then, if we multiply \cref{eq:final_path2_constant} with the pair $(s_1^{(2)},s_2^{(2)})$ swapped, we obtain that 
    \begin{equation*}
        (W(x,y)W(x,z))^d
    \end{equation*}
    is constant as a function of $x\in A$ for almost every fixed $(y,z)\in A^2$.
    
    Denote this constant $>0$ by $\kappa(y,z)$.
    Then, for almost every $x,y,z,w\in A$,
    \begin{align*}
        W(x,y)^{2d}
        &=
        \frac{(W(x,y)W(x,z))^d
        (W(x,y)W(x,w))^d}{(W(x,z)W(x,w))^d}
        =
        \frac{\kappa(y,z)\kappa(y,w)}{\kappa(z,w)}.
    \end{align*}
    Thus, for almost every $y\in A$, the value $W(x,y)$ is constant as a function of $x\in A$; write it as $g(y)$.
    By symmetry, $g(y)=W(x,y)=W(y,x)=g(x)$ almost everywhere on $A^2$, so $g$ is almost everywhere constant.
    Hence, $W$ is almost everywhere constant on $A\times A$ and zero outside $A\times A$.
 \end{proof}

\section{Generalization to hypergraphs}
\label{sec:hypergraphs}
Sidorenko's conjecture also generalizes naturally to hypergraphs, but Sidorenko~\cite{1993:Sidorenko:CorrelationInequalityBipartiteGraphs} himself already pointed out that there are $r$-partite $r$-uniform hypergraphs for which it does not hold, for example the $3$-uniform loose triangle.
Further examples of non-Sidorenko hypergraphs, and the relationship to their extremal numbers, were studied by Conlon, Lee, and Sidorenko~\cite{2024:ConlonEtAl:ExtremalNumbersSidorenko} and by Nie and Spiro~\cite{2023:NieSpiro:SidorenkoHypergraphsRandomTuran}.
Hypergraphs known to be Sidorenko include complete $r$-partite $r$-uniform hypergraphs~\cite{2021:Wigderson:CompleteRPartiteSidorenko}, \emph{loose}\footnote{The loose $r$-uniform hypergraph obtained from a graph $G$ has vertex set $V_G\sqcup ([r-2]\times E_G)$ and edges $\{u,v\}\cup ([r-2]\times\{\{u,v\}\})$ for each edge $\{u,v\}\in E_G$.} expansions of Sidorenko graphs~\cite{2023:NieSpiro:SidorenkoHypergraphsRandomTuran}, and a hypergraph variant of the \emph{one vertex complete to the other side} result~\cite{2025:Lee:SidorenkoExponentsHypergraphs} (also studying Sidorenko exponents to measure how far a hypergraph is from being Sidorenko).
Our \Cref{thm:adjoint_existence} suggests an analogous construction for hypergraphs.
Let $a,b,c,r \in \NN$ with $r \ge 2$ and $F$ an $r$-uniform hypergraph on $[2] \times [a] \sqcup [b] \sqcup [c]$ that is symmetric with respect to the copies of $[a]$.
For a graph $H$, define $L_{F,a,b,c}(H)$ as the $r$-uniform hypergraph on vertex set $[a] \times V_H \sqcup [b] \times E_H \sqcup [c]$ obtained by replacing every edge of $H$ with a copy of $F$, where the copies of $[a]$ are identified over the vertices of $H$ and all copies of $[c]$ are identified.
As in the graph construction, place one copy of $F[{1}\times[a]]$ on $[a]\times{u}$ for every $u\in V_H$, including isolated vertices.
We expect that the construction underlying \Cref{thm:adjoint_existence} extends to $L(H)=L_{F,a,b,c}(H)$.
A candidate right adjoint is suggested by the analogue of \Cref{eq:representation_of_right_adjoint}, but we do not develop the required formalism here.

Such an extension would recover the result of Nie and Spiro~\cite[Corollary~1.12]{2023:NieSpiro:SidorenkoHypergraphsRandomTuran} for loose expansions, with $a=1$, $b=r-2$, $c=0$, and $F$ a single $r$-edge.
Another example for even $r$ is $a=r/2$, $b=c=0$, and $F$ a single $r$-edge, which would imply that so-called \defemph{even}\footnote{For even $r$ the \defemph{even} $r$-uniform hypergraph obtained from $G$ has vertex set $V_G\times [r/2]$ and edges $\{u,v\}\times[r/2]$ for each $\{u,v\} \in E_G$} expansions of Sidorenko graphs are Sidorenko. 
One could also seek to investigate implications of a constant part with $c \neq 0$, or combinations of all of these.
Moreover, this also naturally extends to adjoints defined on $r$-uniform hypergraphs $H$ with new vertices over every $S\subseteq e$ for every $e\in E_H$ (not just $a$ on each original vertex and $b$ on each original edge).

Another interesting application is an extension of the result of Conlon, Fox, and Sudakov~\cite{2010:ConlonEtAl:ApproximateSidorenko} that bipartite graphs with one vertex complete to the other side are Sidorenko, to hypergraphs.
Given pairwise vertex-disjoint Sidorenko $(r-1)$-graphs $S_1,\dots,S_\ell$ and $J_1,\dots,J_t \subseteq [\ell]$ with $J_1=[\ell]$, one can construct an $r$-graph $H$ by adding $t$ new vertices $v_1,\dots,v_t$ and each $v_i$ having link $\bigsqcup_{j \in J_i}S_j$.
In this graph $v_1$ takes the role of the vertex that is complete to the ``other'' side and Lee~\cite[Theorem 1.4]{2025:Lee:SidorenkoExponentsHypergraphs} proved that this $H$ is again Sidorenko.
By including edge colors into our theory (which we avoid because it would make it much more technical), the corresponding framework should allow one to apply the following colored inequality for $(r-1)$-graphs when each $S_i$ is colored with its own color from $c_1,\dots,c_\ell$:
\begin{align}
\label{eq:hypergraph_CFS}
S_1^{(c_1)} \cdots S_\ell^{(c_\ell)} \ge \left(K_{r-1}^{(c_1)}\right)^{e_{S_1}} \cdots \left(K_{r-1}^{(c_\ell)}\right)^{e_{S_\ell}} \, .
\end{align}
Here products and powers denote disjoint unions.
Provided that the edge-colored adjunction and its normalization are established, 
applying the adjoint from colored $(r-1)$- to uncolored $r$-graphs corresponding to the above construction would imply $H \ge K_{r}^{e_H}$ and, hence, that $H$ is Sidorenko.
To see this, note that applying the adjoint to the right-hand side of the displayed inequality gives an $r$-graph as above in which the link of $v_i$ is the disjoint union, over $j\in J_i$, of a matching of size $e_{S_j}$. This reduces to the bipartite graph on
\[
    \{v_1,\ldots,v_t\}
    \sqcup
    \bigsqcup_{j=1}^{\ell}(\{j\}\times[e_{S_j}]),
\]
where $v_i$ is adjacent to $(j,h)$ whenever $j\in J_i$. Since $J_1=[\ell]$, the vertex $v_1$ is complete to the other side, and hence this graph is Sidorenko by~\cite{2010:ConlonEtAl:ApproximateSidorenko}.

\section{Concluding Remarks}\label{sec:concluding}

The principal advantage of the adjoint framework is its modularity. Although the individual applications admit direct combinatorial proofs, the adjoint framework places them within a single modular argument. Our approach is closely related to the auxiliary-graph construction used in \Cref{eq:chain_intro}: when $F=K_{m,m}$, $a=m$, and $b=c=0$, the right adjoint in \Cref{eq:representation_of_right_adjoint} becomes
\[
R(W)(y,z) = \int_{[0,1]^{[m]^2}} \mathbbm 1_F^W((y,z),\psi)\,\mathrm d\mathcal L(\psi) = \prod_{i,j\in[m]} W(y_i,z_j),
\]
which is precisely the corresponding auxiliary graphon. The blow-up calculation above thus illustrates the adjoint identity in a case where the auxiliary graph can be written down directly.

More generally, once the identity
\[
t(L(H),W)
= \int t(H,(R(W)_x,f_x))\,\mathrm d\mu(x)
\]
has been established for some $L=L_{F,a,b,c}$, inequalities and zero-set containments can be transported through $L$ without reconstructing an auxiliary graph for each application. This provides a common proof mechanism for the preservation results concerning blow-ups, subdivisions, Cartesian products, and strong products; the application-specific work is reduced to verifying the relevant normalization and analyzing the equality case.

It would be interesting to determine which adjoints on simple graphs admit a representation of the form in \Cref{thm:adjoint_existence}, and whether a substantially broader class of constructions is needed.
One may also ask which binomial graph inequalities are preserved by these adjoints and what relations this induces between domination exponents~\cite{2022:BlekhermanRaymond:PathForwardTropicalization,2025:BlekhermanEtAl:DominationExponentsPairs}.

\section*{Declaration of AI use}

AI tools were used in a supporting capacity to identify and obtain relevant literature, cross-check citations, improve language, and proofread the manuscript. They were not used to develop the mathematical ideas or arguments presented in this work.

\bibliographystyle{amsplain}
\bibliography{bib}

\end{document}